\documentclass{amsart}

\usepackage{amsmath}
\usepackage{amssymb}
\usepackage{amsopn}
\usepackage{epsfig}
\usepackage{amsfonts}
\usepackage{latexsym}
\usepackage{graphicx}
\usepackage{calrsfs}

\theoremstyle{plain}
\newtheorem{theorem}{Theorem}[section]

\newtheorem{lemma}{Lemma}[section]
\newtheorem{proposition}{Proposition}[section]

\newtheorem{example}{Example}[section]

\newtheorem{definition}{Definition}[section]
\title[]{Self similarity of generalized  Cantor sets}
\author[]{Derong Kong}
\address{School of Mathematical Science, Yangzhou University, Yangzhou 225002, People's Republic of China.}
\email[]{drkong@yzu.edu.cn,\quad dkong.tudelft@gmail.com}

\date{\today}

\newcommand{\R}{\ensuremath{\mathbb{R}}}

\newcommand{\Z}{\ensuremath{\mathbb{Z}}}

\newcommand{\C}{\ensuremath{\mathcal{C}}}
\newcommand{\D}{\ensuremath{\mathcal{D}}}

\newcommand{\Ga}[1]{\Gamma_{\beta,#1}}
\newcommand{\la}{\lambda}
\newcommand{\ga}{\gamma}
\newcommand{\f}{\infty}
\newcommand{\al}{\alpha}
\newcommand{\Om}{\Omega}

\begin{document}

\begin{abstract}
We consider the self-similar structure of a class of  generalized Cantor sets
$$\Ga{\D}=\Big\{\sum_{n=1}^\f d_n\beta^{n}: d_n\in D_n, n\ge 1\Big\},$$
 where  $0<\beta<1$ and $D_n, n\ge 1,$ are nonempty and finite subsets of $\Z$. We give a necessary and sufficient condition   for $\Ga{\D}$ to be a homogeneously generated self similar
  set. An application  to the self-similarity of intersections of
generalized Cantor sets will be given.

\smallskip

\noindent{\em Key words.}{ generalized Cantor sets, self-similar sets, strongly eventually periodic, iterated function systems, intersections.} 

\medskip

\noindent{\bf{MSC}:  28A80, 37B10}

\end{abstract}
\maketitle
\section{Introduction}\label{sec: Introduction}

For an integer $d$, let $\phi_d$ be a contractive  map  defined by
\begin{equation}\label{eq: similitude map}
\phi_d(x)=\beta (x+d), \quad x\in\R,
\end{equation}
where $0< \beta< 1$. Let $\Z^\f$ be the set of infinite  sequences $\{d_n\}_{n=1}^\f$ with each $d_n\in\Z$. We define the coding map $\pi: \Z^\f\rightarrow\R$  by
\begin{equation}\label{eq: projection map}
\pi(\{d_n\}_{n=1}^\f):=\lim_{m\rightarrow\f}\phi_{d_1}\circ\cdots\circ\phi_{d_m}(0)=\sum_{n=1}^\f d_n\beta^n.
\end{equation}
For $n\ge 1$, let $D_n$ be a nonempty and finite subset of  $\Z$, and let
$\D=\bigotimes_{n=1}^\f D_n$
 be the set of infinite sequences $\{d_n\}_{n=1}^\f$ with $d_n\in D_n$ for all $n\ge 1$.
Then the \emph{generalized Cantor set $\Ga{\D}$ of type $\D$} is defined as the image set of $\D$ under the coding map $\pi$, i.e.,
\begin{equation}\label{eq: Gamma_D}
\Ga{\D}:=\pi(\D)=\Big\{\sum_{n=1}^\f d_n\beta^n: d_n\in
D_n~\textrm{for all}~ n\ge 1\Big\}.
\end{equation}
We assume that  the digit sets $D_n, n\ge 1,$ have bounded cardinality  and the sums in (\ref{eq: Gamma_D}) are all convergent, which we express as
\begin{equation}\label{assumptions}
\sup_{n\ge 1}\max_{d, d' \in D_n}(d-d')<\f\quad\textrm{and}\quad \Big|\sum_{i=1}^\f \la_n\beta^n\Big|<\f,
\end{equation}
where $\la_n=\max\{d: d\in D_n\}$.

 The generalized Cantor set $\Ga{\D}$ of type $\D=\bigotimes_{n=1}^\f D_n$ can also be seen in a geometrical way. To illustrate this geometrical construction,  we need the extra assumptions:
\begin{equation*}
\ga_\D:=\inf\{\ga_n: n\ge 1\}>-\f,\quad\la_\D:=\sup\{\la_n: n\ge 1\}<\f,
\end{equation*}
where $\ga_n=\min\{d: d\in D_n\}$.
Then all the digit sets $D_n\subseteq\{\ga_\D, \ga_\D+1,\cdots, \la_\D\}, n\ge 1$.
Let $F_0=[\ga_\D\beta/(1-\beta), \la_\D\beta/(1-\beta)]$ and then, for $n\ge 1$, inductively define
\begin{equation*}
F_n=\bigcup_{d\in D_n}\phi_d(F_{n-1}),
\end{equation*}
where $\phi_d$ is the contractive map defined in (\ref{eq: similitude map}).
Clearly, $\{F_n\}_{n=0}^\f$ is a monotonic decreasing sequence of compact subsets of $\R$. Then the generalized Cantor set $\Ga{\D}$ can be written as
\begin{equation*}
\Ga{\D}=\bigcap_{n=0}^\f F_n.
\end{equation*}
 For example, let $\beta=1/5, D_1=\{0,2\}, D_2=\{1,2\}$ and $D_n=\{0, 1, 2\}$ for all $n\ge 3$. Then $\ga_\D=0, \la_\D=2$, and the first few generations $F_0,\cdots, F_3$ of $\Gamma_{1/5, \D}$ are plotted in Figure \ref{Fig: 1}.
\begin{figure}[h!]
 {\centering \includegraphics[width=12cm]{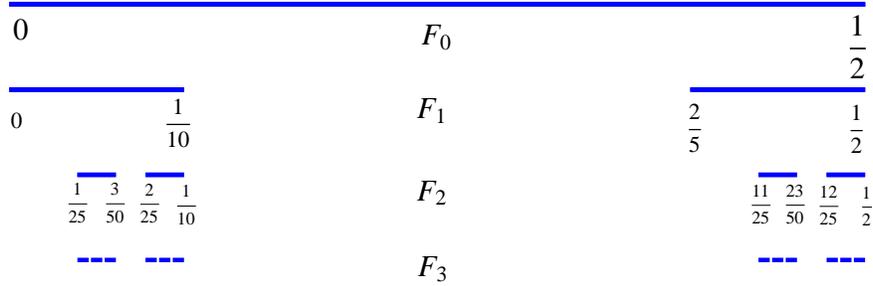}\\
  \caption{The first few generations $F_0, F_1, F_2, F_3$ of the generalized Cantor set $\Ga{\D}$ of type $\D=\bigotimes_{n=1}^\f D_n$ with $\beta=1/5,  D_1=\{0, 2\}, D_2=\{1,2\}$ and $ D_n=\{0, 1, 2\}$ for all $n\ge 3$.}\label{Fig: 1}}
\end{figure}

A map $S: \R\rightarrow\R$ is called a \emph{\index{similitude}similitude}  if there is $r$,  $0<|r|<1$, and $b\in\R$ such that
$
S(x)=r x+b
$ for $x\in\R$. Here $r$ is called the \emph{contraction ratio} of $S$.
Suppose $\{S_j(x)=r_j x+b_j: j\in J\}$, where $J$ is a finite index set. Then
 there exists a unique compact set $\Lambda$ satisfying
\begin{equation*}
\Lambda=\bigcup_{j\in J} S_j(\Lambda).
\end{equation*}
The compact set $\Lambda$ is called\emph{\index{self-similar set}{ a self-similar set }}
generated by the \emph{\index{iterated function system, IFS}{iterated function system} }(IFS)
 $\{S_j: j\in J\}$. In particular, if $r_i=r_j$ for all $i, j\in J$, then $\Lambda$
 is called a \emph{homogeneously generated self-similar set}.

Intersections of Cantor sets have been studied by many authors
(cf.~\cite{Furstenberg_1970, Kenyon_Peres_1991,Kraft_1992,
Moreira_Gustavo_1996, Peres_Solomyak_1998, Elekes_Keleti_Mathe_2010,
Yao_Li_2012}). Recently, Deng et al.~\cite{Deng_He_Wen_2008} studied
the self-similarity of the intersections of the classical
middle-third Cantor set with its translations. Essentially they gave
a characterization for $\Gamma_{1/3,\D}$ being a self-similar set,
where $\D=\bigotimes_{n=1}^\f D_n$ with $D_n$ nonempty subsets of
$\{0,2\}$. Later this   has been extended in
\cite{Kong_Li_Dekking_2010} and \cite{Li_Yao_Zhang_2011} to the case
for
$\Gamma_{\beta,\{0,1,\cdots,N-1\}^\f}\cap(\Gamma_{\beta,\{0,1,\cdots,N-1\}^\f}+t)$,
where $N\ge 2$ is an integer, $0<\beta<1/N$ and $t\in[-1,1]$ has a
unique $\{0,\pm 1,\cdots,\pm(N-1)\}$-code.  In fact they gave a
necessary and sufficient  condition for the generalized Cantor set
$\Ga{\D}$ being a self-similar set, where $0<\beta<1/N$ and   all
the digit sets $D_n\subseteq\{0,1,\cdots,N-1\}, n\ge 1$, are
\emph{\index{consecutive}consecutive}, i.e., there exists some
$\tau_n\ge 0$ such that
 $$D_n=\{\ga_n,\ga_n+1,\cdots, \ga_n+\tau_n\}\subseteq\Z.$$
  However, when some  digit set $D_n$ is not consecutive,
  nothing is known about the self-similarity of $\Ga{\D}$.
   This motivates us to investigate the self-similarity of generalized Cantor sets.

We arrange the paper in the following way. In Section \ref{sec: main
result} we state the main result in Theorem \ref{th: 1}, and its
proof will be given in Section \ref{sec: proof}.  In Section
\ref{sec: application} we consider an application to the
self-similarity of intersections of generalized Cantor sets.

\section{Preliminaries and the  main result}\label{sec: main result}

 For $0<\beta<1$, let $\Ga{\D}$ be  a generalized Cantor set  of type $\D=\bigotimes_{n=1}^\f D_n$. Recall from Equation (\ref{eq: projection map}) and (\ref{eq: Gamma_D}) that $\pi$ is a surjective map from $\D$ to $\Ga{\D}$ defined by
\begin{equation*}
 \pi(\{d_n\}_{n=1}^\f)=\sum_{i=1}^\f d_n\beta^n.
\end{equation*}

 The infinite sequence $\{d_n\}_{n=1}^\f\in\D$ is called a $\D$-code of $\pi(\{d_n\}_{n=1}^\f)\in\Ga{\D}$. We point out that a point  $x\in\Ga{\D}$ may have multiple $\D$-codes. But when $0<\beta<1/N_\D$, the map $\pi$ from $\D$ to $\Ga{\D}$ is bijective and, hence, each point in $\Ga{\D}$ has a unique $\D$-code.
Here $N_\D$ is the \emph{span} of $\D$ defined as
\begin{equation}\label{eq: N_D}
N_\D:=\sup_{n\ge 1}\max_{d, d' \in D_n}(1+d-d').
\end{equation}

Obviously, $N_\D\ge 1$. If $N_\D=1$, then $\Ga{\D}$ contains only a single point. Excluding this trivial case and using the assumption in (\ref{assumptions}), we have $2\le N_\D<\f$. If  no confusion arises about $\D$, we will write $N$ instead of $N_\D$. Let $\Om_N:=\{0,1,\cdots,N-1\}.$

\begin{definition}[Deng, He and Wen \cite{Deng_He_Wen_2008}]\label{def: 1}
A sequence $\{d_n\}_{n=1}^\f\in\Om_N^\f$ is called   \emph{strongly eventually  periodic (or simply, SEP) with period $p\in\Z^+$} if there exist  two finite sequences $\{a_\ell\}_{\ell=1}^p, \{b_\ell\}_{\ell=1}^p\in\Om_N^p$ such that
\begin{equation*}
\{d_n\}_{n=1}^\f=\{a_\ell\}_{\ell=1}^p \overline{\{a_\ell+b_\ell\}_{\ell=1}^p},
\end{equation*}
 where $\overline{\{c_\ell\}_{\ell=1}^p}\in\Om_N^\f$ stands for the infinite repetition  of a finite sequence $\{c_\ell\}_{\ell=1}^p\in\Omega_N^p$.
\end{definition}

  Obviously, a periodic sequence $\{d_n\}_{n=1}^\f$ is a SEP sequence, and a SEP sequence $\{d_n\}_{n=1}^\f$ is eventually periodic. Note that a SEP sequence $\{d_n\}_{n=1}^\f$ is called  strongly periodic  in \cite{Deng_He_Wen_2008}.

  Analogously, we have the definition for the strong eventual periodicity  of a sequence  of sets.

  \begin{definition}\label{def: 2}
A sequence  of sets $\{D_n\}_{i=1}^\f$ with $\emptyset\ne D_n\subseteq\Z$ is called \emph{strongly eventually periodic (or simply, SEP) with period $p\in\Z^+$} if there exist two finite sequences of sets $\{A_\ell\}_{\ell=1}^p, \{B_\ell\}_{\ell=1}^p$  such that
  \begin{equation*}
\{D_n\}_{n=1}^\f=\{A_\ell\}_{\ell=1}^p \overline{\{A_\ell+B_\ell\}_{\ell=1}^p},
\end{equation*}
 where $A+B=\{a+b: a\in A, b\in B\}$  and $\overline{\{C_\ell\}_{\ell=1}^p}$ denotes the infinite repetition  of a finite sequence of sets $\{C_\ell\}_{\ell=1}^p$. \end{definition}

 When a sequence of sets $\{D_n\}_{n=1}^\f$ is SEP with period $p$, it is easy to check that the sequence $\{|D_n|-1\}_{n=1}^\f$ is also SEP with period $p$, where $|A|$ stands for the cardinality of a set $A$. In general this is not true the other way around. But when all the digit sets $D_n, n\ge 1,$ are consecutive, the
SEP of the sequence $\{|D_n|-1\}_{n=1}^\f$ is equivalent to the SEP
of the sequence of the sets $\{D_n-\ga_n\}_{n=1}^\f$, where
$\ga_n=\min\{d: d\in D_n\}$ and $A-b:=A+\{-b\}=\{a-b: a\in A\}$ for
a set $A$ and a real number $b$.

When all the digit sets $D_n, n\ge 1$ are consecutive, the authors
in \cite{Kong_Li_Dekking_2010} and \cite{Li_Yao_Zhang_2011} showed
that for $0<\beta<1/N$ the generalized Cantor set $\Ga{\D}$ is a
self-similar set if and only if the sequence $\{|D_n|-1\}_{n=1}^\f$
is SEP. However, this characterization for the self-similarity of
$\Ga{\D}$ fails  if some digit set $D_n$ is not consecutive. In the
following theorem we give a more general  characterization for
$\Ga{\D}$ being a homogeneously generated self-similar set.

\begin{theorem}\label{th: 1}
  Let $\Ga{\D}$ be a generalized Cantor set of
   type $\D=\bigotimes_{n=1}^\f D_n$  in {\rm(\ref{eq: Gamma_D})} with span $N=N_\D\ge 2$.
   Suppose $0<\beta\le 1/[(3N-1)/2]$. Then  $\Ga{\D}$  is a  homogeneously generated self-similar set
    if, and only if,
 the sequence of sets $\{D_n-\ga_n\}_{n=1}^\f$ is {\rm SEP}, where $\ga_n=\min\{d: d\in D_n\}$.
\end{theorem}

We remark here that Theorem \ref{th: 1} generalize \cite[Theorem
1.2]{Pedersen_Phillips_2012} where the authors only prove the
theorem for $0<\beta<1/(2N-1)$. When $\beta$ gets larger, the proof
is more involved. Moreover, we give an example to illustrate that
the upper bound $1/[(3N-1)/2]$ for $\beta$ in Theorem \ref{th: 1}
can not be improved to $1/N$ as in the consecutive case.

\begin{example}\label{ex: 1}
  Let $D_1=\{0\}, D_2=\{0,4\}, D_{2m+1}=\{0,1\}, D_{2m+2}=\{0,2,4\}$ for all $m\ge 1$. Clearly, the span $N$ of $\D=\bigotimes_{n=1}^\f D_n$ equals $5$, and the sequence of sets $\{D_n\}_{n=1}^\f$ is not {\rm SEP}. Take $\beta=1/6.$ Then $1/[(3N-1)/2]<\beta<1/N$. We will show  that $\Ga{\D}$ is a  self-similar set generated by an IFS $\{g_j(x)=r x+b_j: j\in J\}$ with $r=\beta^2$.

Recall that $\pi$ is the coding map from $\D$ to $\Ga{\D}$ defined
by letting
\begin{equation*}
\pi(\{d_n\}_{n=1}^\f)=\sum_{n=1}^\f d_n\beta^n\quad\textrm{
for}~\{d_n\}_{n=1}^\f\in \D.
 \end{equation*}
 Let $\{g_j(x)=\beta^2 x+b_j\}_{j=1}^8$ be the sequence of similitudes, where
\begin{equation*}
\begin{split}
b_1&=\pi(0000\overline{0}),\quad b_2=\pi(000 2 \overline{0}),\quad b_3=\pi(000 4 \overline{0}),\quad b_4=\pi(0010\overline{0});\\
b_5&=\pi(0400\overline{0}),\quad b_6=\pi(0402\overline{0}),\quad
b_7=\pi(0404\overline{0}),\quad b_8=\pi(0410\overline{0}).
\end{split}
\end{equation*}
Then, by using $D_n =D_{n+2}$ for $n\ge 3$, we have
\begin{equation*}
\begin{split}
g_1(\Ga{\D})&=g_1\Big(\sum_{n=1}^\f D_n \beta^n\Big)=D_1\beta^3+D_2\beta^4+\sum_{n=3}^\f D_n\beta^{n+2}\\
&=D_1\beta^3+D_2\beta^4+\sum_{n=5}^\f D_n\beta^{n}.
\end{split}
\end{equation*}
 Similarly,
$ g_2(\Ga{\D})=D_1\beta^3+(D_2+2)\beta^4+\sum_{n=5}^\f D_n\beta^n,
        g_3(\Ga{\D})=D_1\beta^3+(D_2+4)\beta^4+\sum_{n=5}^\f D_n\beta^n,
            g_4(\Ga{\D})=(D_1+1)\beta^3+D_2\beta^4+\sum_{n=5}^\f D_n\beta^n.$
Then, by using $\beta=1/6$ we obtain that
\begin{equation*}
\bigcup_{j=1}^4 g_j(\Ga{\D})=\sum_{n=3}^\f D_n\beta^n.
\end{equation*}
In a similar way, one can also show that $\bigcup_{j=5}^8
g_j(\Ga{\D})=4\beta^2+\sum_{n=3}^\f D_n\beta^n.$ Hence, $\Ga{\D}$ is
a self-similar set generated by the IFS $\{g_j(x): 1\le j\le 8\}$,
i.e., $ \bigcup_{j=1}^8 g_j(\Ga{\D})=\sum_{n=1}^\f
D_n\beta^n=\Ga{\D}.
$
\end{example}

\section{Proof of Theorem \ref{th: 1}}\label{sec: proof}
For convenience we  introduce the following notations. For a
sequence of sets $\{A_n\}_{n=1}^\f$, let
 $$\sum_{n=1}^\f A_n:=\Big\{\sum_{n=1}^\f a_n: a_n\in A_n, n\ge 1\Big\},$$
 and for a real number $x$ and a set $A$ let $xA:=\{xa: a\in A\}=Ax$. Then by (\ref{eq: Gamma_D}) the generalized Cantor set $\Ga{\D}$ of type $\D$ can be rewritten as $$\Ga{\D}=\sum_{n=1}^\f D_n\beta^n.$$

To prove Theorem  \ref{th: 1}, it is convenient to shift $\Ga{\D}$
such that $0$ is the left endpoint.  More explicitly, let $
D'_n:=D_n-\ga_n$ with $\ga_n=\min\{d: d\in D_n\}. $ Then $0\in
D_n'\subseteq \Om_N$ for all $n\ge 1$. Accordingly, let
\begin{equation*}\label{eq:H-D}
\D':=\bigotimes_{n=1}^\f D'_n=\bigotimes_{n=1}^\f(D_n-\ga_n).
\end{equation*}
 Then the generalized Cantor set $\Ga{\D'}$ of type $\D'=\bigotimes_{n=1}^\f D'_n$
 is a translation of  $\Ga{\D}$, since
\begin{equation}\label{eq:
Gamma_H-D}
\begin{split}
 \Ga{\D'}&=\sum_{n=1}^\f D_n'\beta^n=\sum_{n=1}^\f
(D_n-\gamma_n)\beta^n=\sum_{n=1}^\f
D_n\beta^n-\sum_{n=1}^\f\gamma_n\beta^n\\
&=\Ga{\D}-\sum_{n=1}^\f\gamma_n\beta^n.
\end{split}
\end{equation}
So  it suffices to prove Theorem  \ref{th: 1} for $\Ga{\D'}$ instead
of $\Ga{\D}$. Clearly, since $0\in D'_n$ for all $n\ge 1$, we have
$\sum_{n=1}^m d_n\beta^n=\pi(d_1\cdots d_m \overline{0})\in\Ga{\D'}$
for any $m\ge 1$ and for any $d_\ell\in D'_\ell, \;
\ell=1,\cdots,m$. In particular, $0=\pi(\overline{0})\in\Ga{\D'}$.

The sufficiency of Theorem \ref{th: 1} follows from the following
proposition which can be proved in a similar way as in the proof of
\cite[Theorem 1.2]{Li_Yao_Zhang_2011}.
\begin{proposition}\label{prop: sep to}
 Let $\Ga{\D'}$ be a generalized Cantor set of type
 $\D'=\bigotimes_{n=1}^\f D_n'$ in {\rm(\ref{eq: Gamma_H-D})}. Suppose $0<\beta<1$.
 If the sequence of sets $\{D_n'\}_{n=1}^\f$ is {\rm SEP}, then $\Ga{\D'}$  is a homogeneously generated
 self-similar set.
\end{proposition}
\begin{proof}
By Definition \ref{def: 2}, we assume
$\{D_n\}_{n=1}^\f=\{A_\ell\}_{\ell=1}^q
\overline{\{A_\ell+B_\ell\}_{\ell=1}^q}$ for some $q\in\Z^+$. In a
similar calculation as in \cite[Theorem 1.2]{Li_Yao_Zhang_2011} one
can show that $\Ga{\D'}$ is a self-similar set generated by the IFS
$\{g_e(x)=\beta^q x+e: e\in\mathcal{E}\}$,
  where
  $$\label{eq: sec3 E}
  \mathcal{E}=\bigg\{ \sum_{\ell=1}^{2q}d_\ell\beta^{\ell}: d_\ell\in A_\ell, d_{\ell+q}\in B_\ell~\textrm{for all}~1\le\ell\le q \bigg\}.
$$
This finishes the proof.
\end{proof}

To prove the necessity we need more effort. Recall in {\rm(\ref{eq:
Gamma_H-D})} that $\Ga{{\D'}}$ is a generalized Cantor set of type
${\D'}=\bigotimes_{n=1}^\f D'_n$  with span $N=N_{\D'}\ge 2$. If
$\Ga{\D'}$ is homogeneously generated self-similar set, we will show
in the following lemma that the sequence of sets $\{D'_n\}_{n=1}^\f$
is eventually periodic.
\begin{lemma}\label{lem: th2-1}
 Suppose $0<\beta<1/N$. If $\Ga{{\D'}}$ is a homogeneously generated self-similar set, then there exists some $p\in\Z^+$ such that $\{D'_n\}_{n=1}^\f=\{
D'_\ell\}_{\ell=1}^p\overline{\{ D'_{p+\ell}\}_{\ell=1}^{p}}$  with
$D'_\ell\subseteq D'_{p+\ell}$ for $\ell=1,\cdots,p$.
\end{lemma}
\begin{proof}
Suppose $\Ga{\D'}$ is a self-similar set generated by an IFS
$\{f_i(x)=r x+a_i: i\in I \}$ with $0<|r|<1$, i.e.,
 \begin{equation*}\label{eq: sec4 lem 1}
 \Ga{{\D'}}=\bigcup_{i\in I} f_i(\Ga{{\D'}}).
  \end{equation*}
  One can assume that $0<r<1$, since otherwise we can consider the IFS $\{f_{i}\circ f_{i'}(x):  i, i'\in I\}$ instead of $\{f_{i}(x): i\in I \}$.  Since $0<\beta<1$, there exists some  $\alpha> 0$ such that  $r=\beta^{\al}$.
In a similar way as in the proof of \cite[Theorem
3.2]{Kong_Li_Dekking_2010} one can show that $\al$ must be a
rational number. Take $m\in\Z^+$ such that $q:=m\al\in\Z^+$. Then
$\Ga{\D'}$ can be also generated by $\{g_j: j\in
J\}=\{f_{i_1}\circ\cdots\circ f_{i_m}(x): i_n\in I, n=1,\cdots,m\}$
with contraction ratio equals $r^{m}=\beta^q$.

Note that $0\in\Ga{{\D'}}$. Then there exists $j_0\in J$ such that
$g_{j_0}(x)=\beta^q x$. For $\ell\ge 1$, let $d\in D'_\ell$, and
 we take $d\beta^{\ell}=\pi(0^{\ell-1}d
\;\overline{0})\in\Ga{{\D'}}$. Then
\begin{equation*}
g_{j_0}(d\beta^{\ell})=d\beta^{q+\ell}\in\Ga{{\D'}}.
\end{equation*}
This implies $d\in D'_{q+\ell}$, since, by $0<\beta<1/N$, any point in $\Ga{{\D'}}$ has a unique ${\D'}$-code. So $D'_\ell\subseteq D'_{q+\ell}$.
By iteration, this yields
\begin{equation*}
D'_\ell\subseteq D'_{q+\ell}\subseteq\cdots\subseteq D'_{mq+\ell}\subseteq\cdots
\end{equation*}
for all $1\le \ell\le q$ and $m\ge 1$.

Since $D'_n\subseteq\Om_N$ for all $n\ge 1$,  there exists some large $m_*\ge 1$ which can be chosen independent of $\ell$, such that $D'_{m_* q+\ell}=D'_{mq+\ell}$ for all $1\le \ell\le q$ and $m\ge m_*$. Take $p=m_* q$. Then
\begin{equation*}
\{D'_n\}_{n=1}^\f=\{ D'_\ell\}_{\ell=1}^p \overline{\{ D'_{p+\ell}\}_{\ell=1}^{p}}.
\end{equation*}
Clearly, $D'_\ell\subseteq D'_{p+\ell}$ for $\ell=1,\cdots,p.$ This completes the proof.
\end{proof}

Recall from Section \ref{sec: main result} that an infinite sequence $\{x_n\}_{n=1}^\f\in{\D'}$ is called a ${\D'}$-code of $x\in\Ga{{\D'}}$ if $x=\sum_{n=1}^\f x_n\beta^n$ with $x_n\in D'_n$ for all $n\ge 1$. Since $0<\beta<1/N$, we have $N\le[1/\beta]$. Then ${\D'}\subseteq\Om_N^\f\subseteq\Omega_{[1/\beta]}^\f$, and so $\Ga{\D'}\subseteq\Ga{\Om_{[1/\beta]}^\f}$. In this case, $\{x_n\}_{n=1}^\f$ is also called a $\Omega_{[1/\beta]}^\f$-code of $x$.
\begin{lemma}\label{lem: unique code}
 Suppose $0<\beta<1/N$. If
$x\in\Ga{{\D'}}\subseteq\Ga{\Om_{[1/\beta]}^\f}$ has a
$\Om_{[1/\beta]}^\f$-code $\{x_n\}_{n=1}^\f$ which is not of the
form $d_1\cdots d_k \overline{([1/\beta]-1)}$, then
$\{x_n\}_{n=1}^\f$ is  the unique ${\D'}$-code of $x$.
\end{lemma}
\begin{proof}
  The lemma follows by the fact that when $1/\beta\notin\Z^+$ the coding map $\pi$ from $\Om_{[1/\beta]}^\f$ to $\Ga{\Om_{[1/\beta]}^\f}$ is bijective, and  when $1/\beta\in\Z^+$ then $\pi$ is almost bijective in the sense that only countably many points in $\Ga{\Om_{[1/\beta]}^\f}$ have two $\Om_{[1/\beta]}^\f$-codes of the forms $d_1\cdots d_{k-1}d_k \overline{([1/\beta]-1)}$ and $d_1\cdots d_{k-1} (d_k+1)\overline{ 0}$ for some $d_k+1\le [1/\beta]-1$.
\end{proof}

\begin{proof}[Proof of the necessity]
If $N=2$, then all the digit sets $D_n$ are consecutive, and in this case
the necessity follows from \cite{Li_Yao_Zhang_2011} and
\cite{Kong_Li_Dekking_2010}. In the following we will assume $N\ge
3$. Then by the assumption we have
 $0<\beta\le 1/[(3N-1)/2]<1/N.$ By Lemma \ref{lem: th2-1} there exists $p\in\Z^+$ such that
\begin{equation}\label{eq: periodic}
  \{D_n\}_{n=1}^\f=\{ D'_\ell\}_{\ell=1}^p   \overline{\{ D'_{p+\ell}\}_{\ell=1}^p}\quad\textrm{with}\quad D'_\ell\subseteq D'_{p+\ell}\quad\textrm{for}~1\le\ell\le p.
\end{equation}
Moreover, we can require from Lemma \ref{lem: th2-1} that $\Ga{{\D'}}$ is generated by an IFS $\{g_j(x)=\beta^p x+b_j: j\in J\}$.
Recall that $0\in\Ga{{\D'}}$. Then we have $b_j=g_j(0)\in\Ga{{\D'}}$ for all $j\in J$. Since $0<\beta<1/N$, any point in $\Ga{{\D'}}$ has a unique ${\D'}$-code. Let $\{b_{j,n}\}_{n=1}^\f$ be the unique ${\D'}$-code of $b_j$, i.e.,
\begin{equation*}
b_j=\pi(\{b_{j,n}\}_{n=1}^\f)=\sum_{n=1}^\f b_{j,n}\beta^n\quad\textrm{with}~ b_{j,n}\in D'_n~\textrm{ for all }~n\ge 1.
\end{equation*}

 For $1\le \ell\le p$, let $B_\ell:=\{b\in\Z: b+D'_\ell\subseteq D'_{p+\ell}\}$. By Equation (\ref{eq: periodic}) the necessity will then follow if we can show that  $D'_{p+\ell}=B_\ell+D'_\ell$ for $1\le\ell\le p$. Clearly, $0\in B_\ell$ since $D'_\ell\subseteq D'_{p+\ell}$. Directly from the definition of sum of sets, we have the inclusion
\begin{equation*}
 B_\ell +D'_\ell=\bigcup_{b\in B_\ell}(b+D'_\ell)\subseteq D'_{p+\ell}.
\end{equation*}
 On the other hand, let $d\in D'_{p+\ell}$ for some $1\le \ell\le p$. We split the proof of $d\in B_\ell+D'_\ell$ into the following two cases.

Case I. $d\le[(N-1)/{2}]$. Take
\begin{equation*}
x=d\beta^{p+\ell}+\sum_{n=1}^\f\lambda_{p+\ell+n}\beta^{p+\ell+n}\in\Ga{{\D'}},
\end{equation*}
 where $\la_n:=\max\{d: d\in D'_{n}\}$ for $n\ge 1$.
Since $\Ga{{\D'}}=\bigcup_{j\in J}g_j(\Ga{{\D'}})$, there exist $j_x\in J$  and $x'=\sum_{n=1}^\f x_n'\beta^n=\pi(\{x'_n\}_{n=1}^\f)\in\Ga{{\D'}}$ such that $x=g_{j_x}(x')$, i.e.,
\begin{eqnarray}
  d\beta^{p+\ell}&+&\sum_{n=1}^\f\lambda_{p+\ell+n}\beta^{p+\ell+n}=\sum_{n=1}^p b_{j_x,n}\beta^{n}+\sum_{n=1}^{\ell-1}(b_{j_x, p+n}+x'_n)\beta^{p+n}\nonumber\\
  &&\quad+(b_{j_x,p+\ell}+x'_\ell)\beta^{p+\ell}+\sum_{n=1}^\f (b_{j_x,p+\ell+n}+x'_{\ell+n})\beta^{p+\ell+n}.\label{eq: case 1_1}
\end{eqnarray}
Since $0<\beta<1/N$, we have $x<\beta^{p+\ell-1}$, and then we obtain that
$
b_{j_x,n}=0$  for $1\le n<p+\ell
$
 { and} $ x'_n=0$ { for} $1\le n<\ell$.
 Then Equation (\ref{eq: case 1_1}) can be rearranged in the following way:
\begin{equation}\label{eq: case 1_2}
d=(b_{j_x,p+\ell}+x'_\ell)+\sum_{n=1}^\f (b_{j_x, p+\ell+n}+x_{\ell+n}'-\la_{p+\ell+n})\beta^n.
\end{equation}
Since $0<\beta<1/N$ and the digit $b_{j_x,n}$ satisfy $b_{j_x,n}\le \la_n$ for all $n\ge 1$, we have
\begin{eqnarray*}
\Big|\sum_{n=1}^\f (b_{j_x, p+\ell+n}+x_{\ell+n}'-\la_{p+\ell+n})\beta^n\Big|&\le&
\sum_{n=1}^\f\big|x_{\ell+n}'-(\la_{p+\ell+n}-b_{j_x,p+\ell+n})\big|\beta^n\\
&\le&\sum_{n=1}^\f(N-1)\beta^n<1.
\end{eqnarray*}
Then it follows from Equation (\ref{eq: case 1_2}) that
\begin{equation}\label{eq: prop_2}
d= b_{j_x,p+\ell}+x'_{\ell}\,\in\, b_{j_x,p+\ell}+D'_\ell.
\end{equation}

Since $d\le[({N-1})/{2}]$, we have by Equation (\ref{eq: prop_2}) that $b_{j_x, p+\ell}\le[({N-1})/{2}]$. This, together with  $\la_\ell\le N-1$ and $0<\beta\le 1/[{(3N-1)}/{2}]$, yields
$$b_{j_x,p+\ell}+\la_\ell\le[(3N-1)/{2}]-1\le[ 1/\beta]-1,$$
 i.e.,  $b_{j_x,p+\ell}+D'_\ell\subseteq\Om_{[ 1/\beta ]}$. Note that
\begin{equation}\label{eq: prop_sep}
g_{j_x}(D'_\ell\beta^\ell)=\sum_{n\ne p+\ell}b_{j_x,n}\beta^n+(b_{j_x,p+\ell}+D'_\ell)\beta^{p+\ell}\subseteq\Ga{{\D'}}=\sum_{n=1}^\f D'_n\beta^n.
\end{equation}
Since $0<\beta\le 1/[(3N-1)/2]$ and $N\ge 3$, we have $b_{j_x,n}\le N-1<[1/\beta]-1$ for $n\ge 1$. Then, by using Lemma \ref{lem: unique code} in Equation (\ref{eq: prop_sep}), we obtain  $b_{j_x,p+\ell}+D'_\ell\subseteq D'_{p+\ell}$, i.e.,  $b_{j_x,p+\ell}\in B_\ell$. So, by Equation (\ref{eq: prop_2}) we have $d\in B_\ell+D'_\ell$.

Case II. $d> [({N-1})/{2}]$. Take
\begin{equation*}
y=\la_{p+\ell-1}\beta^{p+\ell-1}+d\beta^{p+\ell}+\sum_{n=1}^\f\la_{p+\ell+n}\beta^{p+\ell+n}\in\Ga{{\D'}}.
\end{equation*}
Similarly, since $\Ga{{\D'}}=\bigcup_{j\in J}g_j(\Ga{{\D'}})$, there exist $j_y\in J$ and $y'=\sum_{n=1}^\f y'_n\beta^n=\pi(\{y'_n\}_{n=1}^\f)\in\Ga{{\D'}}$ such that $y= g_{j_y}(y')$, i.e.,
\begin{equation}\label{eq: case 2_1}
\begin{split}
&\la_{p+\ell-1}\beta^{p+\ell-1}+d\beta^{p+\ell}+\sum_{n=1}^\f\la_{p+\ell+n}\beta^{p+\ell+n}\\
=&~\sum_{n=1}^p b_{j_y,n}\beta^n+\sum_{n=1}^{\ell-2}(b_{j_y,p+n}+y'_n)\beta^{p+n}+(b_{j_y, p+\ell-1}+y'_{\ell-1})\beta^{p+\ell-1}\\
&\quad+(b_{j_y,p+\ell}+y'_\ell)\beta^{p+\ell}+\sum_{n=1}^\f(b_{j_y,p+\ell+n}+y_{\ell+n}')\beta^{p+\ell+n}.
\end{split}
\end{equation}
By using $0<\beta<1/N$ we have $y<\beta^{p+\ell-2}$, and then we obtain that
\begin{equation}\label{eq: b_{j_y}}
b_{j_y,n}=0\quad\textrm{for}~1\le n<p+\ell-1
\end{equation} {and} $y'_n=0$ for $1\le n<\ell-1$. Then Equation (\ref{eq: case 2_1}) can be rearranged as
\begin{eqnarray}
  \la_{p+\ell-1}&=&(b_{j_y,p+\ell-1}+y'_{\ell-1})+(b_{j_y,p+\ell}+y'_\ell-d)\beta\nonumber\\
  &&+\sum_{n=1}^\f(b_{j_y,p+\ell+n}+y_{\ell+n}'-\la_{p+\ell+n})\beta^{n+1},
\label{eq: case 2_2}
\end{eqnarray}
where we set $y_0'=0$. Since $0<\beta\le 1/[ (3N-1)/2]$ and $d>[(N-1)/2]$, we have
\begin{equation}\label{eq: case 2_3}
  -([ 1/\beta ]-1)\le b_{j_y,p+\ell}+y_\ell'-d\le [ (3N-1)/2]-1\le[ 1/\beta]-1.
\end{equation}
In a similar way as in Case I, we can show that
\begin{equation}\label{eq: case 2_4}
  \Big|\sum_{n=1}^\f (b_{j_y,p+\ell+n}+y_{\ell+n}'-\la_{p+\ell+n})\beta^{n+1}\Big|<\beta.
\end{equation}
Then, by using Equation (\ref{eq: case 2_3}) and (\ref{eq: case 2_4}) it follows that
\begin{equation*}
 \Big|(b_{j_y,p+\ell}+y'_\ell-d)\beta+\sum_{n=1}^\f(b_{j_y,p+\ell+n}+y_{\ell+n}'-\la_{p+\ell+n})\beta^{n+1}\Big|<1.
\end{equation*}
 This, together with Equation (\ref{eq: case 2_2}), yields that
\begin{equation}\label{eq: case 2_5}
\la_{p+\ell-1}=b_{j_y,p+\ell-1}+y_{\ell-1}'.
\end{equation}
Substituting (\ref{eq: case 2_5}) in Equation (\ref{eq: case 2_2}) we obtain
\begin{equation*}
d=(b_{j_y,p+\ell}+y'_\ell)+\sum_{n=1}^\f(b_{j_y,p+\ell+n}+y_{\ell+n}'-\la_{p+\ell+n})\beta^{n}.
\end{equation*}
This, again by using Equation (\ref{eq: case 2_4}), yields that
\begin{equation}\label{eq: case 2_6}
d=b_{j_y,p+\ell}+y'_\ell\,\in\, b_{j_y,p+\ell}+D'_\ell.
\end{equation}
If $b_{j_y, p+\ell}+\la_\ell<[ 1/\beta]$, then $b_{j_y, p+\ell}+D'_\ell\subseteq\Om_{[ 1/\beta]}$. In a similar way as in Equation (\ref{eq: prop_sep}), we can show that $b_{j_y,p+\ell}+D'_\ell\subseteq D'_{p+\ell}$, i.e., $b_{j_y,p+\ell}\in B_\ell$. Thus, by Equation (\ref{eq: case 2_6}) we have $d\in B_\ell+ D'_\ell$.

We will finish the proof of Case II by showing that $b_{j_y, p+\ell}+\la_\ell\ge [ 1/\beta]$ will lead to a contradiction.  Take $z=y_{\ell-1}'\beta^{\ell-1}+\la_\ell\beta^\ell\in \Ga{{\D'}}$. Then, by Equation (\ref{eq: b_{j_y}}) and (\ref{eq: case 2_5}), we have
\begin{eqnarray*}
g_{j_y}(z)&=&(b_{j_y, p+\ell-1}+y'_{\ell-1})\beta^{p+\ell-1}+(b_{j_y, p+\ell}+\la_{\ell})\beta^{p+\ell}+\sum_{n>p+\ell}b_{j_y,n}\beta^n\\
&=&\la_{p+\ell-1}\beta^{p+\ell-1}+(b_{j_y, p+\ell}+\la_{\ell})\beta^{p+\ell}+\sum_{n> p+\ell}b_{j_y,n}\beta^n.
\end{eqnarray*}
Since $0<\beta<1/N$ and  $b_{j_y, p+\ell}+\la_\ell\ge [ 1/\beta]$, we have
\begin{eqnarray*}
g_{j_y}(z)&>&\la_{p+\ell-1}\beta^{p+\ell-1}+\sum_{n=p+\ell}^\f(N-1)\beta^n\\ g_{j_y}(z)&<&(\la_{p+\ell-1}+1)\beta^{p+\ell-1}+\sum_{n=p+\ell}^\f(N-1)\beta^n.
\end{eqnarray*}
Since $g_{j_y}(z)\in\Ga{\D'}\subseteq\Ga{\Om_N^\f}$, this implies that there exist $c_n\in D_n', n\ge p+\ell$ such that
\begin{equation}\label{eq: gz-1}
g_{j_y}(z)=(\la_{p+\ell-1}+1)\beta^{p+\ell-1}+\sum_{n=p+\ell}^\f c_n\beta^n\in\Ga{\D'}.
\end{equation}
Since by $N\ge 3$ that $0<\beta<1/[(3N-1)/2]\le 1/(N+1)$, we have $\la_{p+\ell-1}+1\le N\le[1/\beta]-1$. Then, by using Lemma \ref{lem: unique code} in Equation (\ref{eq: gz-1}), we obtain $\la_{p+\ell-1}+1\in D'_{p+\ell-1}$, leading to a contradiction with the definition of $\la_{p+\ell-1}$.
  \end{proof}

\section{Intersections of generalized Cantor sets}\label{sec: application}

Let $\Ga{\C}$ and $\Ga{{\D}}$ be two generalized Cantor sets of types $\C=\bigotimes_{n=1}^\f C_n$ and ${\D}=\bigotimes_{n=1}^\f D_n$, respectively. Using (\ref{eq: Gamma_D}), one can easily write the intersection $\Ga{\C}\cap\Ga{{\D}}$ as
\begin{equation*}
\Ga{\C}\cap\Ga{{\D}}=\Big\{\sum_{n=1}^\f d_n\beta^n: d_n\in C_n\cap D_n\Big\}=\pi(\C\cap{\D}),
\end{equation*}
where
$\C\cap{\D}:=\bigotimes_{n=1}^\f (C_n\cap D_n).$
 So $\Ga{\C}\cap\Ga{{\D}}$ is also a generalized Cantor set if $\Ga{\C}\cap\Ga{{\D}}\ne\emptyset$.
 Using Theorem \ref{th: 1}, we have the following characterization
 for the self-similarity of intersections of generalized Cantor sets.

\begin{proposition}\label{th: intersection}
Let $\Ga{\C}$ and $ \Ga{{\D}}$ be two generalized Cantor sets
 of types $\C=\bigotimes_{n=1}^\f C_n$ and ${\D}=\bigotimes_{n=1}^\f D_n$ respectively.
 Let $N=N_{\C\cap\D}\ge 2$ be the span of $\C\cap\D:=\bigotimes_{n=1}^\f (C_n\cap D_n)$.
 Suppose $0<\beta\le 1/[(3N-1)/2]$. Then  $\Ga{\C}\cap\Ga{{\D}}$ is a homogeneously generated self-similar set
   if, and only if, the sequence of sets $\{C_n\cap D_n-\ga_{n}\}_{n=1}^\f$ is {\rm SEP}, where $\ga_{n}=\min\{d: d\in C_n\cap D_n\}$.
\end{proposition}

In particular, we consider  the self-similarity of intersections of
a generalized Cantor set $\Ga{\D}$ with its translations, i.e.,
$\Ga{\D}\cap(\Ga{\D}+t)$ for  $t\in\R$. Clearly,
\begin{equation*}
\Ga{\D}\cap(\Ga{\D}+t)\neq\emptyset\quad\textrm{if and only if}\quad t\in\Ga{\D}-\Ga{\D}.
\end{equation*}
Using Equation (\ref{eq: Gamma_D}) the difference set $\Ga{\D}-\Ga{\D}$ can be written as
 \begin{equation*}\label{eq: sec5_2}
 \Ga{\D}-\Ga{\D}=\Big\{\sum_{n=1}^\f t_n\beta^{i } : t_n\in D_n-D_n\Big\}=\pi(\D-\D),
 \end{equation*}
 where
 $
\D-\D:=\bigotimes_{n=1}^\f (D_n-D_n).
$
Then, for $t\in\Ga{\D}-\Ga{\D}$ we can show in a similar way as in \cite{Li_Yao_Zhang_2011} that
\begin{equation*}\label{eq: intersection}
\Ga{\D}\cap(\Ga{\D}+t)=\bigcup_{\{t_n\}_{n=1}^\f}\pi\bigg(\bigotimes_{n=1}^\f\big(D_n\cap(D_n+t_n)\big)\bigg),
\end{equation*}
where the union is taken over all $\D-\D$-codes $\{t_n\}_{n=1}^\f$ of $t$. If $t \in\Ga{\D}-\Ga{\D}$ has a unique $\D-\D$-code $\{t_n\}_{n=1}^\f$, then
\begin{equation*}
\Ga{\D}\cap(\Ga{\D}+t)=\pi\bigg(\bigotimes_{n=1}^\f\big(D_n\cap(D_n+t_n)\big)\bigg),
\end{equation*}
which is  a generalized Cantor  set of type $\bigotimes_{n=1}^\f(D_n\cap(D_n+t_n))$.


By using Theorem \ref{th: 1}, we have the following proposition on
the self-similarity of intersections of a generalized Cantor set
with its translations.
\begin{proposition}\label{th: translations}
 Let $\Ga{\D}$ be a generalized Cantor set of type $\D=\bigotimes_{n=1}^\f D_n$
  with span $N_\D\ge 2$ in {\rm(\ref{eq: Gamma_D})}.
  Suppose $0<\beta\le 1/[(3N_\D-1)/2]$ and $t\in\Ga{\D}-\Ga{\D}$ has a unique $\D-\D$-code $\{t_n\}_{n=1}^\f$.
   Then  $\Ga{\D}\cap(\Ga{{\D}}+t)$ is a homogeneously generated self-similar set
   if, and only if,  the sequence of sets
$\{D_n\cap(D_n+t_n)-\ga_{n}(t)\}_{n=1}^\f$ is {\rm SEP}, where
$\ga_{n}(t)=\min\{d: d\in D_n\cap(D_n+t_n)\}$.
\end{proposition}


\bigskip

\begin{center}
{\bf Acknowledgement}
\end{center}

\medskip

%

The author would like to thank Wenxia Li for  his  critical reading of the previous versions of the paper. His comments and suggestions have greatly improved the paper. The author would also  thank  Michel Dekking for many fruitful discussions. The author is partially supported  by the National Natural Science Foundation of China 10971069 and Shanghai Education Committee Project 11ZZ41.

\bibliographystyle{plain}
\bibliography{Self-similar-structure}
\end{document}